\documentclass[12pt,brazil,english,reqno]{amsart}
\usepackage[T1]{fontenc}
\usepackage[latin9]{inputenc}
\usepackage{array}
\usepackage{float}
\usepackage{amsthm}
\usepackage{amstext}
\usepackage{amssymb}

\makeatletter

\providecommand{\tabularnewline}{\\}

\numberwithin{equation}{section}
\numberwithin{figure}{section}

\usepackage{graphicx}
\usepackage{array}\usepackage{pifont}\usepackage[english]{babel}
\usepackage{float}\usepackage{amsthm}\usepackage{fixltx2e}\PassOptionsToPackage{normalem}{ulem}
\usepackage{ulem}

\providecommand{\tabularnewline}{\\}

\usepackage{etex}\usepackage{ifxetex}\usepackage{ifpdf}
\usepackage{chngpage}\usepackage{calc}\usepackage{pstricks}\usepackage{pstricks-add}\usepackage{pst-math}\usepackage{pst-xkey}
\usepackage[all]{xy}

\AtBeginDocument{
  
}
\usepackage{amsthm}\usepackage{xspace}\usepackage{a4wide}\usepackage{amscd}\usepackage{eucal}\usepackage[all]{xy}
 \usepackage{mathrsfs}\usepackage{amsthm}\pretolerance=10000 \tolerance=10000
\numberwithin{equation}{section}

\newtheorem{example}[equation]{Example}
\newtheorem{definition}[equation]{Definition}

\newtheorem{theorem}[equation]{Theorem}
\newtheorem{lemma}[equation]{Lemma}
\newtheorem{proposition}[equation]{Proposition}

\newcommand{\J}{\mathcal{J}}
\newcommand{\ka}{\textbf{\texttt{k}}}
\newcommand{\car}{{\textrm{char }}}

\makeatother

\usepackage{babel}
\begin{document}
\include{comandos} 


\title{Four dimensional Jordan algebras}

\author{Mar\'{\i}a Eugenia Martin}

\address{María Eugenia Martin. \newline Instituto de Matemática e Estat\'{i}stica.
\newline  Universidade de São Paulo.\newline São Paulo, SP, Brasil.}

\email{eugenia@ime.usp.br}
\begin{abstract}
In this paper, we classify four-dimensional Jordan algebras over an
algebraically closed field of characteristic different of two. We
establish the list of $73$ non-isomorphic Jordan algebras. 
\end{abstract}
\maketitle

\section{Introduction}

The goal of this paper is to obtain an algebraic classification of
Jordan algebras of dimension four over algebraically closed field
$\ka$ such that $\car\ka\neq2$. Our motivation to obtain such algebras
comes out from the intention to study the variety of four-dimensional
Jordan algebras. In order to understand this variety we need the list
of all Jordan algebras, associative and nonassociative, unitary and
non-unitary.

The classification problem of the algebraic structures of given dimension,
i.e. find the list of all non-isomorphic objects in this structure,
has been extensively studied. The only class of algebras (associative,
Jordan and Lie) which is completely described is the one of simple
algebras. In general, the list of all algebras is known only for small
dimensions. In 1975, P. Gabriel presented in \cite{finiterepresentatiotypeisopen}
the lists of all unitary associative algebras over an algebraically
closed field of dimension $n$ for $n\leq4$. Four years later, G.
Mazzola, in his work \cite{mazzola}, classified $5$-dimensional
unitary associative algebras over an algebraically closed field and
in \cite{happel} D. Happel obtained the list of unitary associative
algebras of dimension $6$. For Lie algebras the classification is
known for dimension till $6$, see \cite{kirillov}. H. Wesseler in
\cite{wesseler} described unitary Jordan algebras over algebraically
closed field up to dimension $6$. In particular he showed that all
such algebras are special. In 1989, H. Sherkulov in \cite{TeseRuso}
classified nonassociative Jordan algebras up to dimension $4$ and
in 2011, in their article \cite{ancocheabermudes}, Ancochea Bermudez
and others classified the laws of three-dimensional and four-dimensional
nilpotent Jordan algebras over the field of complex numbers.

In this work we generalize these papers and classify Jordan algebras
(both unital and nonunital, associative and nonassociative) of dimension
four over algebraically closed field of characteristic different of
two. In the forthcoming paper we will use this description to study
deformations between Jordan algebras and describe the variety $Jor_{4}$.

The paper is organized as follows. In Section \ref{sec:Preliminaries},
we recall the basic concepts and necessary results for finite-dimensional
Jordan algebras. In Section \ref{sec:Indecomposable-Jordan-algebras},
we list all indecomposable Jordan algebras of dimension less than
$4$. In Section \ref{sec:4-Jordan-algebras}, we describe non-isomorphic
Jordan algebras of dimension four and, finally, in section \ref{sec:Remarks}
we show that all algebras are pairwise nonisomorphic.

\section{Preliminaries\label{sec:Preliminaries}}

In this section we present the basic concepts, the notation and principal
results about Jordan algebras.

We will work only over an algebraically closed field $\ka$ of characteristic
$\neq2$ and, furthermore, all Jordan algebras are assumed to be finite
dimensional over $\ka$.

\begin{definition} A \textbf{Jordan $\ka$-algebra} is an algebra
$\J$ with a multiplication \textquotedbl{}$\cdot$\textquotedbl{}
satisfying the following identities for any $x,$ $y\in\J$: 
\begin{align}
x\cdot y & =y\cdot x\mbox{,}\nonumber \\
((x\cdot x)\cdot y)\cdot x & =(x\cdot x)\cdot(y\cdot x).\label{eq:identidadejor}
\end{align}
\end{definition} We will also use a linearization of Jordan identity
\eqref{eq:identidadejor} 
\begin{equation}
(x,y,z\cdot w)+(w,y,z\cdot x)+(z,y,x\cdot w)=0\mbox{,}\label{eq:linearjord}
\end{equation}
for any $x,\mbox{ }y,\mbox{ }z,\mbox{ }w\in\J$. Here $(x,y,z):=(x\cdot y)\cdot z-x\cdot(y\cdot z)$
is the associator of $x,\, y,\, z$.

Let $\mathcal{A}$ be an associative algebra. We define on the underlying
vector space of the algebra $\mathcal{A}$ a new operation $\odot$
of multiplication given by the formula
\[
x\odot y=\frac{1}{2}(x\cdot y+y\cdot x),
\]
for any $x$, $y\in\mathcal{A}$, where $x\cdot y$ denotes the multiplication
in $\mathcal{A}$. The algebra $\left(\mathcal{A},\odot\right)$ is
a Jordan algebra and is denoted by $\mathcal{A}^{\left(+\right)}$.\begin{definition}
A Jordan algebra $\J$ is called \textbf{special} if there exists
an associative algebra $\mathcal{A}$ such that $\J$ is a Jordan
subalgebra of $\mathcal{A}^{\left(+\right)}$. Jordan algebras which
are not special are called \textbf{exceptional}. \end{definition}

\begin{example}Let $\left(\mathcal{U},j\right)$ be an associative
algebra with involution and
\[
H(\mathcal{U},j)=\left\{ u\in\mathcal{U}\mid u=j(u)\right\} ,
\]
be the set of elements symmetric with respect to $j$. Then $\left(H\left(\mathcal{U},j\right),\odot\right)$is
a subalgebra of $\mathcal{U}^{\left(+\right)}$ and thus $H(\mathcal{U},j)$
is a special Jordan algebra.\end{example}

In $\J$ we define inductively a series of subsets by setting 
\begin{gather*}
\J^{1}=\J^{\left\langle 1\right\rangle }=\J\text{,}\\
\J^{n}=\J^{n-1}\cdot\J+\J^{n-2}\cdot\J^{2}+\cdots+\J\cdot\J^{n-1},\\
\J^{\left\langle n\right\rangle }=\J^{\left\langle n-1\right\rangle }\cdot\J.
\end{gather*}
The subset $\J^{n}$ is called the \textbf{$n$-th power of the algebra
$\J$}. 

The chain of subsets $\J^{\left\langle 1\right\rangle }\supseteq\J^{\left\langle 2\right\rangle }\supseteq\cdots\supseteq\J^{\left\langle n\right\rangle }\supseteq\cdots$
is a chain of ideals of the algebra $\J$ and it is called the \textbf{lower
central series of $\J$}.

Since $\J^{i}=\J^{\left\langle i\right\rangle }$ for $i=1,2,3$ we
will use only the $i$-th power notation in these cases.

\begin{definition} A Jordan algebra $\J$ is said to be \textbf{nilpotent}
if there exists an integer $n\in\mathbb{N}$ such that $\J^{\left\langle n\right\rangle }=0$.
The minimum $n$ for which this condition holds is the \textbf{nilindex}
of $\J$. \end{definition}

If $s$ is the nilindex of $\J$, we define the \textbf{nilpotency
type} of the algebra $\J$ as the sequence $(n_{1},n_{2},n_{3},\cdots,n_{s-1})$,
where $n_{i}=\dim\left(\J^{\left\langle i\right\rangle }/\J^{\left\langle i+1\right\rangle }\right)$.

We remark that all $n_{i}>0$. In fact, suppose that there exists
an $i\in\mathbb{N}$, $1\leq i\leq s-1$ such that $n_{i}=0$ then
$\dim\J^{\left\langle i\right\rangle }=\dim\J^{\left\langle i+1\right\rangle }$
and since $\J^{\left\langle i+1\right\rangle }\subseteq\J^{\left\langle i\right\rangle }$
we have $\J^{\left\langle i\right\rangle }=\J^{\left\langle i+1\right\rangle }$.
Consequently $\J^{\left\langle i+2\right\rangle }=\J^{\left\langle i\right\rangle }\cdot\J=\J^{\left\langle i\right\rangle }$,
by induction $\J^{\left\langle i\right\rangle }=\J^{\left\langle k\right\rangle }$
for every $k\in\mathbb{N},\mbox{ }k>i$. In particular it holds for
$k=s$, so $\J^{\left\langle i\right\rangle }=\J^{\left\langle s\right\rangle }=0$
which is impossible since $s$ is the nilindex of $\J$. \begin{definition}
A Jordan algebra $\J\neq0$ is called \textbf{simple} if $0$ and
$\J$ itself are the only ideals of $\J$. A Jordan algebra is \textbf{semisimple}
if it is a direct sum of simple algebras.\end{definition} \begin{proposition}
\label{prop:quocienteSS}\cite[V.2, V.5 and VII.6]{jacobson}Every
finite-dimensional Jordan algebra has a unique maximal nilpotent ideal
which is called the \textbf{radical} of $\J$, $N=\operatorname*{Rad}(\J)$,
and the quotient $\J/N$ is semisimple. Moreover $\J_{ss}:=\J/N$
has an identity element and its simple decomposition is unique. If
$\ka=\overline{\ka}$ then $\J=\J_{ss}\oplus N$ (as vector spaces).
\end{proposition} We recall the Albert Theorem that classifies all
finite-dimensional simple Jordan algebras. \begin{theorem} \label{thm:albert}\cite[Corollary V.6.2]{jacobson}Let
$\J$ be a finite-dimensional simple Jordan algebra over an algebraically
closed field $\ka$. Then we have the following possibilities for
$\J$: 
\begin{enumerate}
\item $\J=\ka$, 
\item $\J=\J(V,f)$ the Jordan algebra of a non-degenerate symmetric bilinear
form $f$ in a finite-dimensional vector space $V$ such that $\dim V>1$, 
\item $\J=H(M_{n}(\mathcal{D}),J),\mbox{ }n\geq3$, where $(\mathcal{D},j)$
is a composition algebra with involution $j$, of dimension $1,\mbox{ }2$
or $4$ if $n\geq4$ and of dimension $1,\mbox{ }2,\mbox{ }4$, and
$8$ if $n=3$. And $J$ is the standard involution in $M_{n}(\mathcal{D})$
associated with the involution $j$, i.e. $J(X)$ is the matrix obtained
from the matrix $X\in M_{n}(\mathcal{D})$ by applying the involution
$j$ to each entry in $X$ and then transposing.
\end{enumerate}
\noindent \end{theorem} \begin{theorem} \label{thm:Peirce1}\cite[Chap. III.1]{jacobson}Let
$\J$ be a Jordan algebra with idempotent $e$. Then $\J$ decomposes
into a direct sum or Peirce components 
\[
\mathcal{J}=\mathcal{J}_{1}\oplus\mathcal{J}_{\frac{1}{2}}\oplus\mathcal{J}_{0},
\]
where $\mathcal{J}_{i}=\{x\in\mathcal{J}\mid x\cdot e=ix\}$, for
$i=0,\mbox{ }\frac{1}{2},\mbox{ }1$. This decomposition is called
\textbf{ Peirce decomposition} \textbf{of $\J$ relative to the idempotent}
$e$. The multiplication table for the Peirce decomposition is: 
\begin{equation}
\begin{array}{c}
\mathcal{J}_{1}^{2}\subseteq\mathcal{J}_{1},\qquad\mathcal{J}_{1}\cdot\mathcal{J}_{0}=0,\qquad\mathcal{J}_{0}^{2}\subseteq\mathcal{J}_{0}\text{,}\\
\mathcal{J}_{0}\cdot\mathcal{J}_{\frac{1}{2}}\subseteq\mathcal{J}_{\frac{1}{2}},\qquad\mathcal{J}_{1}\cdot\mathcal{J}_{\frac{1}{2}}\subseteq\mathcal{J}_{\frac{1}{2}},\qquad\mathcal{J}_{\frac{1}{2}}^{2}\subseteq\mathcal{J}_{0}\oplus\mathcal{J}_{1}\text{.}
\end{array}\label{eq:pierce_one}
\end{equation}
 \end{theorem} This theorem has the following generalization: if
$\mathcal{J}$ is a Jordan algebra with an identity element $1=\sum_{i=1}^{n}e_{i}$
which decomposes into a sum of pairwise orthogonal idempotents $e_{i}$,
we have the \textbf{Peirce decomposition} \textbf{of $\J$ relative
to idempotents }$\{e_{1},\ldots,e_{n}\}$: 
\begin{equation}
\mathcal{J}=\bigoplus_{1\leq i\leq j\leq n}\mathcal{J}_{ij}\label{eq:decopierce}
\end{equation}
where $\mathcal{J}_{ii}=\left\{ x\in\mathcal{J}\mid x\cdot e_{i}=x\right\} $
and $\mathcal{J}_{ij}=\left\{ x\in\mathcal{J}\mid x\cdot e_{i}=x\cdot e_{j}=\frac{1}{2}x\right\} $.
The multiplication table for the Peirce decomposition is:

\noindent 
\begin{equation}
\begin{array}{c}
\mathcal{J}_{ii}^{2}\subseteq\mathcal{J}_{ii},\qquad\mathcal{J}_{ij}\cdot\mathcal{J}_{ii}\subseteq\mathcal{J}_{ij},\qquad\mathcal{J}_{ij}^{2}\subseteq\mathcal{J}_{ii}\oplus\mathcal{J}_{jj}\text{,}\\
\mathcal{J}_{ij}\cdot\mathcal{J}_{jk}\subseteq\mathcal{J}_{ik},\qquad\mathcal{J}_{ii}\cdot\mathcal{J}_{jj}=\mathcal{J}_{ii}\cdot\mathcal{J}_{jk}=\mathcal{J}_{ij}\cdot\mathcal{J}_{kl}=0,
\end{array}\label{eq:pierce_many}
\end{equation}
where the indexes $i$, $j$, $k$, $l$ are all different.

\section{Indecomposable Jordan algebras of small dimensions\label{sec:Indecomposable-Jordan-algebras}}

In this section we present the lists of all one- two- and three dimensional
indecomposable Jordan algebras. Henceforth, for convenience we drop
$\cdot$ and denote a multiplication in $\J$ simply as $xy$. 

By Proposition \ref{prop:quocienteSS} every Jordan algebra $\J$,
can be decomposed as $\J=\J_{ss}\oplus N$, where $N$ is the radical
of $\J$ and $\J_{ss}$ is semisimple. We will denote by $e_{i}$
the elements in $\J_{ss}$ and by $n_{i}$ the ones which belong to
$N$.

\subsection{Jordan algebras of dimension one\label{sec:apenddim1}}

There are two non-isomorphic one-dimensional Jordan algebras: the
simple algebra $\mathcal{F}_{1}=\ka e$, with $e^{2}=e$ and the nilpotent
algebra $\mathcal{F}_{2}=\ka n$, with $n^{2}=0$.

\subsection{Jordan algebras of dimension $2$\label{sec:apenddim2}}

The only semisimple Jordan algebra of dimension $2$ is $\ka e_{1}\oplus\ka e_{2}$.
If $\dim N=1$ then $\J$ has an idempotent $e_{1}$, $N$ is generated
by $n_{1}$ and the action of $e_{1}n_{1}=in_{1}$, $i=0,\frac{1}{2},1$
defines three non-isomorphic algebras. Finally there are two nilpotent
algebras: algebra with zero multiplication and algebra generated by
$n_{1}$, with $n_{1}^{3}=0$. Therefore we obtain the following three
indecomposable two-dimensional Jordan algebras:

\begin{table}[H]
\begin{tabular}{|c|>{\centering}m{6cm}|>{\centering}m{5cm}|}
\hline 
$\J$ & Multiplication Table & Observation\tabularnewline
\hline 
\hline 
$\mathcal{B}_{1}$ & \vspace{0.2cm}
$e_{1}^{2}=e_{1}\quad e_{1}n_{1}=n_{1}\quad n_{1}^{2}=0$\vspace{0.2cm}
 & associative\tabularnewline
\hline 
$\mathcal{B}_{2}$ & \vspace{0.2cm}
$e_{1}^{2}=e_{1}\quad e_{1}n_{1}=\frac{1}{2}n_{1}\quad n_{1}^{2}=0$\vspace{0.2cm}
 & nonassociative\tabularnewline
\hline 
$\mathcal{B}_{3}$ & \vspace{0.2cm}
${n_{1}}^{2}=n_{2}\quad n_{1}n_{2}=0\quad{n_{2}}^{2}=0$\vspace{0.2cm}
 & nilpotent, associative\tabularnewline
\hline 
\end{tabular}

\caption{Indecomposable two-dimensional Jordan algebras.}

\end{table}

\subsection{Jordan algebras of dimension $3$\label{sec:apenddim3}}

In \cite{irynashesta} all three dimensional Jordan algebras are described.
We obtain the following $10$ indecomposable algebras:

\begin{table}[H]
\begin{tabular}{|c|>{\centering}m{7cm}|>{\centering}m{3cm}|}
\hline 
$\J$ & Multiplication Table & Observation\tabularnewline
\hline 
\hline 
$\mathcal{T}_{1}$ & \vspace{0.1cm}
$e_{1}^{2}=e_{1}\quad n_{1}^{2}=n_{2}\quad n_{2}^{2}=0$

$e_{1}n_{1}=n_{1}\quad e_{1}n_{2}=n_{2}\quad n_{1}n_{2}=0$ \vspace{0.1cm}
 & unitary, 

associative\tabularnewline
\hline 
$\mathcal{T}_{2}$ & \vspace{0.1cm}

$e_{1}^{2}=e_{1}\quad n_{1}^{2}=0\quad n_{2}^{2}=0$

$e_{1}n_{1}=n_{1}\quad e_{1}n_{2}=n_{2}\quad n_{1}n_{2}=0$ \vspace{0.1cm}
 & unitary, associative\tabularnewline
\hline 
$\mathcal{T}_{3}$ & \vspace{0.1cm}
$n_{1}^{2}=n_{2}\quad n_{2}^{2}=0\quad n_{3}^{2}=0$

$n_{1}n_{2}=n_{3}\quad n_{1}n_{3}=0\quad n_{2}n_{3}=0$ \vspace{0.1cm}
 & associative, nilpotent\tabularnewline
\hline 
$\mathcal{T}_{4}$ & \vspace{0.1cm}
$n_{1}^{2}=n_{2}\quad n_{2}^{2}=0\quad n_{3}^{2}=0$

$n_{1}n_{2}=0\quad n_{1}n_{3}=n_{2}\quad n_{2}n_{3}=0$ \vspace{0.1cm}
 & associative, nilpotent\tabularnewline
\hline 
$\mathcal{T}_{5}$ & \vspace{0.1cm}
$e_{1}^{2}=e_{1}\quad e_{2}^{2}=e_{2}\quad e_{3}^{2}=e_{1}+e_{2}$

$e_{1}e_{2}=0\quad e_{1}e_{3}=\frac{1}{2}e_{3}\quad e_{2}e_{3}=\frac{1}{2}e_{3}$\vspace{0.1cm}
 & unitary, nonassociative, semisimple\tabularnewline
\hline 
$\mathcal{T}_{6}$ & \vspace{0.1cm}
$e_{1}^{2}=e_{1}\quad n_{1}^{2}=0\quad n_{2}^{2}=0$

$e_{1}n_{1}=\frac{1}{2}n_{1}\quad e_{1}n_{2}=n_{2}\quad n_{1}n_{2}=0$
\vspace{0.1cm}
 & nonassociative\tabularnewline
\hline 
$\mathcal{T}_{7}$ & \vspace{0.1cm}
$e_{1}^{2}=e_{1}\quad n_{1}^{2}=0\quad n_{2}^{2}=0$

$e_{1}n_{1}=\frac{1}{2}n_{1}\quad e_{1}n_{2}=\frac{1}{2}n_{2}\quad n_{1}n_{2}=0$\vspace{0.1cm}
 & nonassociative\tabularnewline
\hline 
$\mathcal{T}_{8}$ & \vspace{0.1cm}
$e_{1}^{2}=e_{1}\quad n_{1}^{2}=n_{2}\quad n_{2}^{2}=0$

$e_{1}n_{1}=\frac{1}{2}n_{1}\quad e_{1}n_{2}=0\quad n_{1}n_{2}=0$
\vspace{0.1cm}
 & nonassociative\tabularnewline
\hline 
$\mathcal{T}_{9}$ & \vspace{0.1cm}
$e_{1}^{2}=e_{1}\quad n_{1}^{2}=n_{2}\quad n_{2}^{2}=0$

$e_{1}n_{1}=\frac{1}{2}n_{1}\quad e_{1}n_{2}=n_{2}\quad n_{1}n_{2}=0$
\vspace{0.1cm}
 & nonassociative\tabularnewline
\hline 
$\mathcal{T}_{10}$ & \vspace{0.1cm}
$e_{1}^{2}=e_{1}\quad e_{2}^{2}=e_{2}\quad n_{1}^{2}=0$

$e_{1}e_{2}=0\quad e_{1}n_{1}=\frac{1}{2}n_{1}\quad e_{2}n_{1}=\frac{1}{2}n_{1}$
\vspace{0.1cm}
 & unitary, nonassociative\tabularnewline
\hline 
\end{tabular}

\caption{Indecomposable three-dimensional Jordan algebras.}

\end{table}

\section{Jordan algebras of dimension $4$\label{sec:4-Jordan-algebras}}

In this section we will describe all Jordan algebras of dimension
four. The description is organized according to the dimension of the
radical and subsequently the possible values of the nilpotency type.
Also for each algebra we calculate the dimension of its automorphism
group $\operatorname*{Aut}(\J)$, the annihilator $\operatorname*{Ann}(\J)=\left\{ a\in\J\mid a\J=0\right\} $
and the second power $\J^{2}$. 

We denote by $\J^{\#}=\J\oplus\ka1$ the Jordan algebra obtained by
formal adjoining of the identity element $1$ of $\ka$.

\subsection{Semisimple Jordan algebras}

By Theorem \ref{thm:albert}, there exists only $3$ simple Jordan
algebras of dimension $\leq4$: $\ka e$ of dimension $1$; $\operatorname*{Sym}\nolimits$$_{2}(\ka)^{(+)}$
of dimension $3$ (here denoted by $\mathcal{T}_{5}$) and $M_{2}(\ka)^{(+)}$
of dimension $4$. Then there are three semisimple Jordan algebras
of dimension $4$, namely:

\begin{table}[H]
\begin{tabular}{|c|>{\centering}m{5cm}|>{\centering}m{1.3cm}|>{\centering}m{1.3cm}|>{\centering}m{1cm}|>{\centering}m{3cm}|}
\hline 
$\J$ & Multiplication Table & $\dim$ $\operatorname*{Aut}(\J)$ & $\dim$ $\operatorname*{Ann}(\J)$ & $\dim$ $\J^{2}$ & Observation\tabularnewline
\hline 
$\mathcal{J}_{1}$ & $\mathcal{T}_{5}\oplus\ka e_{4}$ & $1$ & $0$ & $4$ & unitary,

semisimple\tabularnewline
\hline 
$\mathcal{J}_{2}$ & \vspace{0.2cm}
$e_{1}^{2}=e_{1}\quad e_{2}^{2}=e_{2}$ 

$e_{1}e_{3}=\frac{1}{2}e_{3}\quad e_{1}e_{4}=\frac{1}{2}e_{4}$

$e_{2}e_{3}=\frac{1}{2}e_{3}\quad e_{2}e_{4}=\frac{1}{2}e_{4}$

$e_{3}e_{4}=\frac{1}{2}(e_{1}+e_{2})$\vspace{0.2cm}
 & $3$ & $0$ & $4$ & unitary,

semisimple,

it is $M_{2}(\ka)^{(+)}$\tabularnewline
\hline 
$\mathcal{J}_{3}$ & $\ka e_{1}\oplus\ka e_{2}\oplus\ka e_{3}\oplus\ka e_{4}$ & $0$ & $0$ & $4$ & unitary,

semisimple,

associative\tabularnewline
\hline 
\end{tabular}

\caption{Four-dimensional semisimple Jordan algebras}
\end{table}

\subsection{Jordan algebras with one-dimensional radical}

Since $\dim N=1$, $\mathcal{J}_{ss}$ is $3$-dimensional and by
Theorem \ref{thm:albert} we have the following possibilities:

\vspace{0.1cm}

\noindent \textbf{1)$\mathcal{J}_{ss}=\ka e_{1}\oplus\ka e_{2}\oplus\ka e_{3}$}.
Then $\mathcal{J}^{\#}=\mathcal{J}\oplus\ka1$ contains $4$ orthogonal
idempotents $e_{1}$, $e_{2}$, $e_{3}$ and $e_{0}=1-e_{1}-e_{2}-e_{3}$,
so using Peirce decomposition \eqref{eq:decopierce} we have: 
\[
\J=\J_{00}\oplus\J_{01}\oplus\J_{02}\oplus\J_{03}\oplus\J_{11}\oplus\J_{12}\oplus\J_{13}\oplus\J_{22}\oplus\J_{23}\oplus\J_{33},
\]
and the correspondent decomposition of $N$: 
\[
N=N_{00}\oplus N_{01}\oplus N_{02}\oplus N_{03}\oplus N_{11}\oplus N_{12}\oplus N_{13}\oplus N_{22}\oplus N_{23}\oplus N_{33},
\]
where $N_{ij}=N\cap\J_{ij}$. Let $n_{1}$ be a basis of $N$, then
$\J$ is completely defined by Pierce subspace $n_{1}\in N_{ij}$.
Thus we obtain the following pairwise nonisomorphic algebras:

\begin{table}[H]
\begin{tabular}{|c|>{\centering}m{5cm}|>{\centering}m{1.3cm}|>{\centering}m{1.3cm}|>{\centering}m{1cm}|>{\centering}m{3cm}|}
\hline 
$\J$ & Multiplication Table & $\dim$ $\operatorname*{Aut}(\J)$ & $\dim$ $\operatorname*{Ann}(\J)$ & $\dim$ $\J^{2}$ & Observation\tabularnewline
\hline 
$\mathcal{J}_{4}$  & $\mathcal{B}_{1}\oplus\ka e_{2}\oplus\ka e_{3}$  & $1$ & $0$ & $4$ & unitary, associative,

$n_{1}\in N_{11}$\tabularnewline
\hline 
$\mathcal{J}_{5}$  & $\ka e_{1}\oplus\ka e_{2}\oplus\ka e_{3}\oplus\ka n_{1}$  & $1$ & $1$ & $3$ & associative, $n_{1}\in N_{00}$\tabularnewline
\hline 
$\mathcal{J}_{6}$  & \vspace{0.2cm}
$\mathcal{B}_{2}\oplus\ka e_{2}\oplus\ka e_{3}$ \vspace{0.2cm}
 & $2$ & $0$ & $4$ & $n_{1}\in N_{01}$\tabularnewline
\hline 
$\mathcal{J}_{7}$ & $\mathcal{T}_{10}\oplus\ka e_{3}$  & $2$ & $0$ & $4$ & unitary, $n_{1}\in N_{12}$\tabularnewline
\hline 
\end{tabular}

\caption{Four-dimensional Jordan algebras with one-dimensional radical and
semisimple part $\J_{ss}=\ka e_{1}\oplus\ka e_{2}\oplus\ka e_{3}$.}

\end{table}

\noindent \textbf{2)$\mathcal{J}_{ss}=\mathcal{T}_{5}$}. Then the
algebra $\mathcal{J}^{\#}=\mathcal{J}\oplus\ka1$ contains $3$ orthogonal
idempotent elements $e_{1}$, $e_{2}$, $e_{0}=1-e_{1}-e_{2}$ and
we have the following decomposition of $N$ 
\[
N=N_{00}\oplus N_{01}\oplus N_{02}\oplus N_{11}\oplus N_{12}\oplus N_{22}.
\]
Let $n_{1}$ be a basis of $N$ then the action of idempotents is
defined by $i$ and $j$ where $n_{1}\in N_{ij}$. One has to check
how $e_{3}\in\J_{12}$ acts in $n_{1}$. If $n_{1}\in N_{00}$, $e_{3}n_{1}\in\mathcal{J}_{12}\mathcal{J}_{00}=0$
is zero. If $n_{1}\in N_{12}$, then $e_{3}n_{1}\in\mathcal{J}_{12}^{2}\subseteq\mathcal{J}_{11}\oplus\mathcal{J}_{22}$,
and since $e_{3}n_{1}\in N=N_{12}\subseteq\mathcal{J}_{12}$ we have
$e_{3}n_{1}=0$. If $n_{1}\in N_{0i}$ or $n_{1}\in N_{ii}$, $i=1,2$
we obtain that $e_{3}n_{1}=0$ but there exists elements of $\J$
that do not satisfy the identity \eqref{eq:linearjord} 
\[
\begin{array}{c}
(n_{1}e_{3},e_{3},e_{2})+(n_{1}e_{2},e_{3},e_{3})+(e_{3}e_{2},e_{3},n_{1})\neq0\qquad\mbox{if }n_{1}\in N_{01}\mbox{ or }n_{1}\in N_{11},\\
(n_{1}e_{3},e_{3},e_{1})+(n_{1}e_{1},e_{3},e_{3})+(e_{3}e_{1},e_{3},n_{1})\neq0\qquad\mbox{if }n_{1}\in N_{02}\mbox{ or }n_{1}\in N_{22}.
\end{array}
\]
Thus we obtain the following pairwise nonisomorphic algebras: 

\begin{table}[H]
\begin{tabular}{|c|>{\centering}m{5cm}|>{\centering}m{1.3cm}|>{\centering}m{1.3cm}|>{\centering}m{1cm}|>{\centering}m{3cm}|}
\hline 
$\J$ & Multiplication Table & $\dim$ $\operatorname*{Aut}(\J)$ & $\dim$ $\operatorname*{Ann}(\J)$ & $\dim$ $\J^{2}$ & Observation\tabularnewline
\hline 
$\mathcal{J}_{8}$ & \vspace{0.4cm}

$\mathcal{T}_{5}\oplus\ka n_{1}$

\vspace{0.4cm}
 & $2$ & $1$ & $3$ & $n_{1}\in N_{00}$\tabularnewline
\hline 
$\mathcal{J}_{9}$ & \vspace{0.4cm}

$e_{1}^{2}=e_{1}\quad e_{2}^{2}=e_{2}$

$e_{3}^{2}=e_{1}+e_{2}$

$e_{1}e_{3}=\frac{1}{2}e_{3}\quad e_{1}n_{1}=\frac{1}{2}n_{1}$

$e_{2}e_{3}=\frac{1}{2}e_{3}\quad e_{2}n_{1}=\frac{1}{2}n_{1}$

\vspace{0.4cm}
  & $4$ & $0$ & $4$ & unitary, $\ensuremath{n_{1}\in N_{12}}$\tabularnewline
\hline 
\end{tabular}

\caption{Four-dimensional Jordan algebras with one-dimensional radical and
semisimple part $\J_{ss}=\mathcal{T}_{5}$.}

\end{table}

\subsection{Jordan algebras with two-dimensional radical}

The only semi-simple two-dimensional Jordan algebra is $\J_{ss}=\ka e_{1}\oplus\ka e_{2}$,
therefore the unitary algebra $\mathcal{J}^{\#}=\mathcal{J}\oplus\ka1$
contains $3$ orthogonal idempotent elements $e_{1}$, $e_{2}$, $e_{0}=1-e_{1}-e_{2}$
thus we have the following Peirce decomposition of $N$: 
\[
N=N_{00}\oplus N_{01}\oplus N_{02}\oplus N_{11}\oplus N_{12}\oplus N_{22}\text{.}
\]
The ideal $N$ can have two nilpotency types: $(2)$ or $(1,1).$

\vspace{0.1cm}

\noindent \textbf{1) Nilpotency type $(2)$.} Then $N^{2}=0$. Let
$n_{1}\mbox{, }n_{2}$ be a basis of $N$, then it is enough to choose
to which Pierce components belong $n_{1}$ and $n_{2}$. Thus we obtain
the following pairwise nonisomorphic algebras: 

\begin{table}[H]
\begin{tabular}{|c|>{\centering}m{5cm}|>{\centering}m{1.3cm}|>{\centering}m{1.3cm}|>{\centering}m{0.8cm}|>{\centering}m{4cm}|}
\hline 
$\J$ & Multiplication Table & $\dim$ $\operatorname*{Aut}(\J)$ & $\dim$ $\operatorname*{Ann}(\J)$ & $\dim$ $\J^{2}$ & Observation\tabularnewline
\hline 
$\mathcal{J}_{10}$ & \vspace{0.2cm}
$\mathcal{B}_{2}\oplus\ka e_{2}\oplus\ka n_{2}$ \vspace{0.2cm}
 & $3$ & $1$ & $3$ & $n_{1}\in N_{01},\mbox{ }n_{2}\in N_{00}$\tabularnewline
\hline 
$\mathcal{J}_{11}$ & \vspace{0.2cm}
$\mathcal{T}_{10}\oplus\ka n_{2}$\vspace{0.2cm}
 & $3$ & $1$ & $3$ & $n_{1}\in N_{12}\text{, }n_{2}\in N_{00}$\tabularnewline
\hline 
$\mathcal{J}_{12}$ & \vspace{0.2cm}
$\mathcal{T}_{7}\oplus\ka e_{2}$\vspace{0.2cm}
 & $6$ & $0$ & $4$ & $n_{1}\text{, }n_{2}\in N_{01}$\tabularnewline
\hline 
$\mathcal{J}_{13}$ & \vspace{0.2cm}
$\mathcal{B}_{2}\oplus\mathcal{B}_{2}$\vspace{0.2cm}
 & $4$ & $0$ & $4$ & $n_{1}\in N_{02}\text{, }n_{2}\in N_{01}$\tabularnewline
\hline 
$\mathcal{J}_{14}$ & \vspace{0.2cm}
$\mathcal{T}_{6}\oplus\ka e_{2}$\vspace{0.2cm}
 & $3$ & $0$ & $4$ & $n_{1}\in N_{01}\text{, }n_{2}\in N_{11}$\tabularnewline
\hline 
$\mathcal{J}_{15}$ & \vspace{0.2cm}
$\mathcal{B}_{2}\oplus\mathcal{B}_{1}$\vspace{0.2cm}
 & $3$ & $0$ & $4$ & $n_{1}\in N_{11}\text{, }n_{2}\in N_{02}$\tabularnewline
\hline 
$\mathcal{J}_{16}$ & \vspace{0.2cm}
$e_{1}^{2}=e_{1}\quad e_{2}^{2}=e_{2}$

$e_{1}n_{1}=\frac{1}{2}n_{1}$

$e_{1}n_{2}=\frac{1}{2}n_{2}\quad e_{2}n_{1}=\frac{1}{2}n_{1}$ \vspace{0.2cm}
 & $4$ & $0$ & $4$ & $n_{1}\in N_{12}\text{, }n_{2}\in N_{01}$\tabularnewline
\hline 
$\mathcal{J}_{17}$ & \vspace{0.2cm}
$e_{1}^{2}=e_{1}\quad e_{2}^{2}=e_{2}$

$e_{1}n_{1}=\frac{1}{2}n_{1}$

$e_{1}n_{2}=n_{2}\quad e_{2}n_{1}=\frac{1}{2}n_{1}$\vspace{0.2cm}
 & $3$ & $0$ & $4$ & unitary, $n_{1}\in N_{12}\text{, }n_{2}\in N_{11}$\tabularnewline
\hline 
$\mathcal{J}_{18}$ & \vspace{0.2cm}
$e_{1}^{2}=e_{1}\quad e_{2}^{2}=e_{2}$ $e_{1}n_{1}=\frac{1}{2}n_{1}\quad e_{1}n_{2}=\frac{1}{2}n_{2}$
$e_{2}n_{1}=\frac{1}{2}n_{1}\quad e_{2}n_{2}=\frac{1}{2}n_{2}$\vspace{0.2cm}
 & $6$ & $0$ & $4$ & unitary,

$n_{1}\text{, }n_{2}\in N_{12}$\tabularnewline
\hline 
$\mathcal{J}_{19}$ & $\ka e_{1}\oplus\ka n_{1}\oplus\ka e_{2}\oplus\ka n_{2}$ & $4$ & $2$ & $2$ & associative, $n_{1}\text{, }n_{2}\in N_{00}$\tabularnewline
\hline 
$\mathcal{J}_{20}$ & $\mathcal{B}_{1}\oplus\ka e_{2}\oplus\ka n_{2}$ & $2$ & $1$ & $3$ & associative, $n_{1}\in N_{11}\text{, }n_{2}\in N_{00}$\tabularnewline
\hline 
$\mathcal{J}_{21}$ & $\mathcal{T}_{2}\oplus\ka e_{2}$ & $4$ & $0$ & $4$ & unitary, associative, $n_{1}\text{, }n_{2}\in N_{11}$\tabularnewline
\hline 
$\mathcal{J}_{22}$ & $\mathcal{B}_{1}\oplus\mathcal{B}_{1}$ & $2$ & $0$ & $4$ & unitary, associative, $n_{1}\in N_{11}\text{, }n_{2}\in N_{22}$\tabularnewline
\hline 
\end{tabular}

\caption{Four-dimensional Jordan algebras with two-dimensional radical of type
$(2)$. }
\end{table}

\noindent \textbf{2)} \textbf{Nilpotency type $(1,1)$.} There exists
$n\in N$ such that $N=\ka n\oplus\ka n^{2}$ with $n^{3}=0$. Suppose
that $N=N_{ij}$ then $n^{2}\in N_{ij}^{2}\subseteq N_{ii}\oplus N_{jj}$,
that implies $i=j$, therefore $N=N_{ii},\mbox{ }i=0,1,2$. Now, suppose
that $N=N_{ij}\oplus N_{kl}$ with $(i,j)\neq(k,l)$ and $\dim N_{ij}=\dim N_{kl}=1$.
We will analyze this cases, in which we consider $i,\mbox{ }j,\mbox{ }k$
all different. If $N=N_{ii}\oplus N_{ij}$, we can choose $n$ as
an element of $N_{ij}$. In fact if $N_{ii}=\ka a$ and $N_{ij}=\ka b$,
then we have $b^{2}\in N_{ii}\oplus N_{jj}\mbox{ but }N_{jj}=0\mbox{, thus }b^{2}=\alpha a$,
for some $\alpha\in\ka$. Note that $\alpha\neq0$ since by nilpotency
$a^{2}=ab=0$. Consequently $N=\ka b^{2}\oplus\ka b$.

For any other pair of indexes $i,\, j$ and $k,\, l$ we will get
a contradiction to the fact that the nilpotency type of $N$ is $(1,1)$.

\selectlanguage{brazil}%
\renewcommand{\theenumi}{\roman{enumi}}
\selectlanguage{english}%
\begin{enumerate}
\item If $N=N_{ii}\oplus N_{jj}$, then by nilpotency $N_{ii}^{2}=N_{jj}^{2}=0$.
Moreover $N_{ii}N_{jj}=0$ and therefore $N^{2}=0$ leads to contradiction.
\item Analogously if $N=N_{ii}\oplus N_{jk}$, again by nilpotency $N_{ii}^{2}=0$
and by multiplication table \eqref{eq:pierce_many} $N_{jk}^{2}=N_{ii}N_{jk}=0$,
which is a contradiction.
\item Finally if $N=N_{ij}\oplus N_{ik}$ then $N^{2}\subseteq N_{ii}\oplus N_{jj}\oplus N_{jk}\oplus N_{kk}=0.$
\end{enumerate}
Therefore we obtain the following pairwise nonisomorphic algebras,
where $n_{1}^{2}=n_{2}$.

\begin{table}[H]
\begin{tabular}{|c|>{\centering}m{5cm}|>{\centering}m{1.3cm}|>{\centering}m{1.3cm}|>{\centering}m{1cm}|>{\centering}m{3cm}|}
\hline 
$\J$ & Multiplication Table & $\dim$ $\operatorname*{Aut}(\J)$ & $\dim$ $\operatorname*{Ann}(\J)$ & $\dim$ $\J^{2}$ & Observation\tabularnewline
\hline 
$\mathcal{J}_{23}$ & $\mathcal{T}_{8}\oplus\ka e_{2}$ & $2$ & $1$ & $4$ & $n_{1}\in N_{01}\text{,}$

$n_{2}\in N_{00}$\tabularnewline
\hline 
$\mathcal{J}_{24}$ & $\mathcal{T}_{9}\oplus\ka e_{2}$ & $2$ & $0$ & $4$ & $n_{1}\in N_{01}\text{,}$

$n_{2}\in N_{11}$\tabularnewline
\hline 
$\mathcal{J}_{25}$ & \vspace{0.3cm}
$e_{1}^{2}=e_{1}\quad e_{2}^{2}=e_{2}$

$e_{1}n_{2}=n_{2}\quad e_{1}n_{1}=\frac{1}{2}n_{1}$

$e_{2}n_{1}=\frac{1}{2}n_{1}\quad n_{1}^{2}=n_{2}$\vspace{0.3cm}
  & $2$ & $0$ & $4$ & unitary,

$n_{1}\in N_{12}\text{,}$

$n_{2}\in N_{11}$\tabularnewline
\hline 
$\mathcal{J}_{26}$ & $\mathcal{B}_{3}\oplus\ka e_{1}\oplus\ka e_{2}$ & $2$ & $1$ & $3$ & associative,

$n_{1},\mbox{ }n_{2}\in N_{00}$\tabularnewline
\hline 
$\mathcal{J}_{27}$ & $\mathcal{T}_{1}\oplus\ka e_{2}$ & $2$ & $0$ & $4$ & unitary,

associative,

$n_{1},\mbox{ }n_{2}\in N_{11}$\tabularnewline
\hline 
\end{tabular}

\caption{Four-dimensional Jordan algebras with two-dimensional radical of type
$(1,1)$.}

\end{table}

\subsection{Jordan algebras with three-dimensional radical}

In this case, we only have one idempotent $e_{1}\in\J$ and the Pierce
decomposition of $N$ with respect to $e_{1}$ is 
\[
N=N_{0}\oplus N_{\frac{1}{2}}\oplus N_{1}.
\]

\noindent The radical $N$ may have the following nilpotency types:
$(3)$, $(1,1,1)$ or $(2,1)$.

We prove the following lemma that will be frequently used in the rest
of the classification. \begin{lemma}\label{lemma} If $\dim N_{i}=1$
then $N_{i}^{2}=0$, for $i=0,1$. If $\dim N_{\frac{1}{2}}=1$ then
$N_{i}N_{\frac{1}{2}}=0$, for $i=0,1$. \end{lemma} \begin{proof}
Since $N_{i}^{2}\subseteq N_{i}$, by nilpotency this inclusion is
strict $N_{i}^{2}\subsetneq N_{i}$. By hypothesis we have $\dim N_{i}=1$
thus $N_{i}^{2}=0$. For the second part note that $N_{\frac{1}{2}}N_{i}\subseteq N_{\frac{1}{2}}$.
Suppose that the equality holds, then 
\[
0=N^{\left\langle 4\right\rangle }\supseteq((N_{\frac{1}{2}}N_{i})N_{i})N_{i}=N_{\frac{1}{2}}\neq0,
\]
thus $N_{i}N_{\frac{1}{2}}\subsetneq N_{\frac{1}{2}}$. Finally, by
hypothesis we have $\dim N_{\frac{1}{2}}=1$, and it follows that
$N_{i}N_{\frac{1}{2}}=0$.\end{proof}

\noindent \textbf{1)} \textbf{Nilpotency type $(3)$.} Then $N^{2}=0$.
Let $n_{1}\mbox{, }n_{2}$ and $n_{3}$ be a basis of $N$ then the
action of $N$ is defined by choosing to which Pierce subspaces belong
$n_{i}$'s. Thus we obtain the following pairwise nonisomorphic algebras:

\begin{table}[H]
\begin{tabular}{|c|>{\centering}m{5cm}|>{\centering}m{1.3cm}|>{\centering}m{1.3cm}|>{\centering}m{1cm}|>{\centering}m{4cm}|}
\hline 
$\J$ & Multiplication Table & $\dim$ $\operatorname*{Aut}(\J)$ & $\dim$ $\operatorname*{Ann}(\J)$ & $\dim$ $\J^{2}$ & Observation\tabularnewline
\hline 
$\mathcal{J}_{28}$ & \vspace{0.3cm}
$\mathcal{B}_{2}\oplus\ka n_{2}\oplus\ka n_{3}$\vspace{0.3cm}
 & $6$ & $2$ & $2$ & $n_{1}\in N_{\frac{1}{2}}\mbox{, }n_{2}\text{, }n_{3}\in N_{0}$\tabularnewline
\hline 
$\mathcal{J}_{29}$ & $\mathcal{T}_{6}\oplus\ka n_{3}$ & $4$ & $1$ & $3$ & $n_{1}\in N_{\frac{1}{2}}\text{, }n_{2}\in N_{1}\text{,}$ $n_{3}\in N_{0}$\tabularnewline
\hline 
$\mathcal{J}_{30}$ & \vspace{0.3cm}
$\mathcal{T}_{7}\oplus\ka n_{3}$ \vspace{0.3cm}
 & $7$ & $1$ & $3$ & $n_{1}\text{, }n_{2}\in N_{\frac{1}{2}}\text{, }n_{3}\in N_{0}$\tabularnewline
\hline 
$\mathcal{J}_{31}$ & \vspace{0.3cm}
$e_{1}^{2}=e_{1}\quad e_{1}n_{1}=n_{1}$

$e_{1}n_{2}=n_{2}\quad e_{1}n_{3}=\frac{1}{2}n_{3}$\vspace{0.3cm}
 & $6$ & $0$ & $4$ & $n_{1}\text{, }n_{2}\in N_{1}\text{, }n_{3}\in N_{\frac{1}{2}}$\tabularnewline
\hline 
$\mathcal{J}_{32}$ & \vspace{0.3cm}
$e_{1}^{2}=e_{1}\quad e_{1}n_{1}=\frac{1}{2}n_{1}$ $e_{1}n_{2}=\frac{1}{2}n_{2}\quad e_{1}n_{3}=n_{3}$\vspace{0.3cm}
 & $7$ & $0$ & $4$ & $n_{1}\mbox{, }n_{2}\in N_{\frac{1}{2}}\text{, }n_{3}\in N_{1}$\tabularnewline
\hline 
$\mathcal{J}_{33}$ & \vspace{0.3cm}
$e_{1}^{2}=e_{1}\quad e_{1}n_{1}=\frac{1}{2}n_{1}$ $e_{1}n_{2}=\frac{1}{2}n_{2}\quad e_{1}n_{3}=\frac{1}{2}n_{3}$\vspace{0.3cm}
 & $12$ & $0$ & $4$ & $n_{1}\text{, }n_{2}\text{, }n_{3}\in N_{\frac{1}{2}}$\tabularnewline
\hline 
$\mathcal{J}_{34}$ & $\ka e_{1}\oplus\ka n_{1}\oplus\ka n_{2}\oplus\ka n_{3}$ & $9$ & $3$ & $1$ & $n_{1}\text{, }n_{2}\text{, }n_{3}\in N_{0}$,

associative\tabularnewline
\hline 
$\mathcal{J}_{35}$ & $\mathcal{B}_{1}\oplus\ka n_{2}\oplus\ka n_{3}$ & $5$ & $2$ & $2$ & $n_{1}\in N_{1}\text{, }n_{2}\mbox{, }n_{3}\in N_{0}$

associative\tabularnewline
\hline 
$\mathcal{J}_{36}$ & \vspace{0.3cm}
$e_{1}^{2}=e_{1}\quad e_{1}n_{1}=n_{1}$

$e_{1}n_{2}=n_{2}\quad e_{1}n_{3}=n_{3}$\vspace{0.3cm}
 & $9$ & $0$ & $4$ & $n_{1}\text{, }n_{2}\text{, }n_{3}\in N_{1}$,

unitary,

associative\tabularnewline
\hline 
$\mathcal{J}_{37}$ & $\mathcal{T}_{2}\oplus\ka n_{3}$ & $5$ & $1$ & $3$ & $n_{1}\mbox{, }n_{2}\in N_{1}\text{, }n_{3}\in N_{0}$

associative\tabularnewline
\hline 
\end{tabular}

\caption{Four-dimensional Jordan algebras with three-dimensional radical of
type $(3)$.}
\end{table}

\noindent \textbf{2) Nilpotency type $(1,1,1)$.} In this case, $N=\mathcal{T}_{3}$
and we have two non-isomorphic Jordan algebras:

\begin{table}[H]
\begin{tabular}{|c|>{\centering}m{5cm}|>{\centering}m{1.3cm}|>{\centering}m{1.3cm}|>{\centering}m{1cm}|>{\centering}m{3cm}|}
\hline 
$\J$  & Multiplication Table & $\dim$ $\operatorname*{Aut}(\J)$ & $\dim$ $\operatorname*{Ann}(\J)$ & $\dim$ $\J^{2}$ & Observation\tabularnewline
\hline 
$\mathcal{J}_{38}$ & \vspace{0.3cm}
$\mathcal{T}_{3}\oplus\ka e_{1}$\vspace{0.3cm}
 & $3$ & $1$ & $3$ & $N=N_{0}$,

associative\tabularnewline
\hline 
$\mathcal{J}_{39}$ & \vspace{0.3cm}
$e_{1}^{2}=e_{1}\quad n_{1}^{2}=n_{2}$

$e_{1}n_{1}=n_{1}\quad e_{1}n_{2}=n_{2}$

$e_{1}n_{3}=n_{3}\quad n_{1}n_{2}=n_{3}$\vspace{0.3cm}
 & $3$ & $0$ & $4$ & $N=N_{1}$,

unitary,

associative\tabularnewline
\hline 
\end{tabular}

\caption{Four-dimensional Jordan algebras with three-dimensional radical of
type $(1,1,1)$.}
\end{table}

All the other cases lead to contradiction, namely:

\selectlanguage{brazil}%
\renewcommand{\theenumi}{\roman{enumi}}
\selectlanguage{english}%
\begin{enumerate}
\item $N=N_{\frac{1}{2}}$, in this case $N^{2}\subseteq N_{0}\oplus N_{1}=0$. 
\item $N=N_{0}\oplus N_{1}$, $\dim N_{0}=2$. By Lemma \ref{lemma} and
multiplication table \eqref{eq:pierce_one}, $N^{2}\subseteq N_{0}$.
Since $\dim N^{2}=2=\dim N_{0}$, $N^{2}=N_{0}$. Consequently, there
exists $n_{0}\in N_{0}$ such that $N^{3}=\ka n_{0}$, choosing some
$m_{0}\in N_{0}$ and $n_{1}\in N_{1}$ we obtain $N^{2}=N_{0}=\ka n_{0}+\ka m_{0}$.
Hence, we have the following products in $N$: 
\begin{gather*}
N^{\left\langle 4\right\rangle }\ni n_{0}^{2}=n_{0}m_{0}=0,\\
n_{1}^{2}=n_{0}n_{1}=m_{0}n_{1}=0,
\end{gather*}
leaving only $m_{0}^{2}$ nonzero, which contradicts to the fact that
$\dim N^{2}=2$. Analogously we obtain a contradiction if $N=N_{0}\oplus N_{1}$,
with $\dim N_{1}=2$.
\item $N=N_{0}\oplus N_{\frac{1}{2}}$, $\dim N_{0}=2$. By Lemma \ref{lemma},
$N^{2}\subseteq N_{0}$. Since $\dim N^{2}=2=\dim N_{0}$, we have
$N^{2}=N_{0}$. Consequently, there exists $n_{0}\in N_{0}$ such
that $N^{3}=\ka n_{0}$, choosing some $m_{0}\in N_{0}$ and $n_{\frac{1}{2}}\in N_{\frac{1}{2}}$
we obtain $N^{2}=N_{0}=\ka n_{0}+\ka m_{0}$. Hence, we have the following
products in $N$: 
\begin{gather*}
N^{\left\langle 4\right\rangle }\ni n_{0}^{2}=n_{0}m_{0}=0,\qquad N^{3}\ni m_{0}^{2}=\alpha n_{0}\\
n_{0}n_{\frac{1}{2}}=m_{0}n_{\frac{1}{2}}=0\mbox{,}\qquad N^{2}\ni n_{\frac{1}{2}}^{2}=\alpha'n_{0}+\beta'm_{0}.
\end{gather*}
We obtain a Jordan algebra only if $\alpha=0$ or $\beta'=0$, but
in both cases it contradicts the fact that $\dim N^{2}=2$. Analogous
proof leads to contradiction if $N=N_{1}\oplus N_{\frac{1}{2}}$,
with $\dim N_{1}=2$.
\item $N=N_{0}\oplus N_{\frac{1}{2}}$, $\dim N_{\frac{1}{2}}=2$. First
we show that 
\begin{equation}
(N_{\frac{1}{2}}N_{0})N_{0}=(N_{0}N_{\frac{1}{2}})N_{\frac{1}{2}}=0\label{eq:1200}
\end{equation}
.

\noindent By \eqref{eq:pierce_one} $(N_{\frac{1}{2}}N_{0})N_{0}\subseteq N_{\frac{1}{2}}N_{0}\subseteq N_{\frac{1}{2}}$.
If $N_{\frac{1}{2}}N_{0}=N_{\frac{1}{2}}$ then 
\[
0=N^{\left\langle 4\right\rangle }\supseteq((N_{\frac{1}{2}}N_{0})N_{0})N_{0}=N_{\frac{1}{2}}\neq0\mbox{,}
\]
therefore $N_{\frac{1}{2}}N_{0}\subsetneq N_{\frac{1}{2}}$. Analogously,
if $(N_{\frac{1}{2}}N_{0})N_{0}=N_{\frac{1}{2}}N_{0}$ then we obtain
\[
0=N^{\left\langle 4\right\rangle }\supseteq((N_{\frac{1}{2}}N_{0})N_{0})N_{0}=N_{\frac{1}{2}}N_{0},
\]
thus $N^{2}=N_{0}^{2}+N_{\frac{1}{2}}^{2}\subseteq N_{0}$ but $\dim N^{2}=2$
while $\dim N_{0}=1.$ Therefore $(N_{\frac{1}{2}}N_{0})N_{0}\subsetneq N_{\frac{1}{2}}N_{0}\subsetneq N_{\frac{1}{2}}$
and, consequently, 
\[
\dim[(N_{\frac{1}{2}}N_{0})N_{0}]<\dim(N_{\frac{1}{2}}N_{0})<\dim N_{\frac{1}{2}}=2,
\]
thus $\dim(N_{\frac{1}{2}}N_{0})=1$ and $(N_{\frac{1}{2}}N_{0})N_{0}=0$.

By \eqref{eq:pierce_one} $(N_{0}N_{\frac{1}{2}})N_{\frac{1}{2}}\subseteq N_{\frac{1}{2}}^{2}\subseteq N_{0}$.
If $(N_{0}N_{\frac{1}{2}})N_{\frac{1}{2}}=N_{0}$ then 
\[
0=N^{\left\langle 4\right\rangle }\supseteq((N_{0}N_{\frac{1}{2}})N_{\frac{1}{2}})N_{\frac{1}{2}}=N_{0}N_{\frac{1}{2}}\neq0\mbox{,}
\]
which contradicts to $\dim(N_{0}N_{\frac{1}{2}})=1$. Consequently,
$(N_{0}N_{\frac{1}{2}})N_{\frac{1}{2}}\subsetneq N_{0}$ and, since
$\dim N_{0}=1$, we have $(N_{0}N_{\frac{1}{2}})N_{\frac{1}{2}}=0$.

Then, since $N=\mathcal{T}_{3}$, for the basis element $n_{1}$ we
have $n_{1}=m_{0}+m_{\frac{1}{2}}$, with $m_{i}\in N_{i},$ $i=0,\frac{1}{2}$
then $n_{1}^{2}=m'_{0}m_{\frac{1}{2}}+m_{\frac{1}{2}}^{2}$ and $n_{1}^{3}=m_{\frac{1}{2}}^{3}$,
using that $N_{\frac{1}{2}}^{2}\subseteq N_{0}$, Lemma \ref{lemma}
and the identities \eqref{eq:1200}.

Hence, 
\[
(m_{\frac{1}{2}}^{2},e_{1},m_{\frac{1}{2}})=(m_{\frac{1}{2}}^{2}e_{1})m_{\frac{1}{2}}-m_{\frac{1}{2}}^{2}(e_{1}m_{\frac{1}{2}})=-\frac{1}{2}n_{1}^{3}\neq0,
\]
and then $\J$ is not a Jordan algebra.

\noindent An analogous proof shows that there is no four-dimensional
Jordan algebra with $N=N_{1}\oplus N_{\frac{1}{2}}$, $\dim N_{\frac{1}{2}}=2$. 

\item $N=N_{0}\oplus N_{\frac{1}{2}}\oplus N_{1}$ where $\dim N_{i}=1$
for $i=0,\,\frac{1}{2},\,1$. By Lemma \ref{lemma} and \eqref{eq:pierce_one}
we have $N^{2}\subseteq N_{0}\oplus N_{1}$ that implies $N^{3}\subseteq(N_{0}\oplus N_{1})N=0$
which is a contradiction. 
\end{enumerate}
\vspace{0.1cm}

\textbf{3)} \textbf{Nilpotency type $(2,1)$}. Then or $N=\mathcal{B}_{3}\oplus\ka n_{3}$
or $N=\mathcal{T}_{4}$. We have the following cases depending on
the dimension of the subspaces $N_{i}$ for $i=0,\,\frac{1}{2},\,1$.
\begin{enumerate}
\item $N=N_{0}$. We obtain two nonisomorphic associative algebras:

\begin{table}[H]
\begin{tabular}{|c|>{\centering}m{5cm}|>{\centering}m{1.3cm}|>{\centering}m{1.3cm}|>{\centering}m{1cm}|>{\centering}m{3cm}|}
\hline 
$\J$ & Multiplication Table & $\dim$ $\operatorname*{Aut}(\J)$ & $\dim$ $\operatorname*{Ann}(\J)$ & $\dim$ $\J^{2}$ & Observation\tabularnewline
\hline 
$\mathcal{J}_{40}$ & $\mathcal{B}_{3}\oplus\ka e_{1}\oplus\ka n_{3}$  & $5$ & $2$ & $2$ & $N=\mathcal{B}_{3}\oplus\ka n_{3}$ associative\tabularnewline
\hline 
$\mathcal{J}_{41}$ & $\mathcal{T}_{4}\oplus\ka e_{1}$ & $4$ & $1$ & $2$ & $N=\mathcal{T}_{4}$ associative\tabularnewline
\hline 
\end{tabular}

\caption{Four-dimensional Jordan algebras with three-dimensional radical of
type $(2,1)$ and $N=N_{0}$.}
\end{table}

\item $N=N_{1}$. We obtain two nonisomorphic associative algebras:

\begin{table}[H]
\begin{tabular}{|c|>{\centering}m{5cm}|>{\centering}m{1.3cm}|>{\centering}m{1.3cm}|>{\centering}m{1cm}|>{\centering}m{3cm}|}
\hline 
$\J$ & Multiplication Table & $\dim$ $\operatorname*{Aut}(\J)$ & $\dim$ $\operatorname*{Ann}(\J)$ & $\dim$ $\J^{2}$ & Observation\tabularnewline
\hline 
$\mathcal{J}_{42}$ & \vspace{0.2cm}
$e_{1}^{2}=e_{1}\quad n_{1}^{2}=n_{2}$

$e_{1}n_{1}=n_{1}\quad e_{1}n_{2}=n_{2}$

$e_{1}n_{3}=n_{3}$\vspace{0.2cm}
 & $5$ & $0$ & $4$ & $N=\mathcal{B}_{3}\oplus\ka n_{3}$,

unitary,

associative\tabularnewline
\hline 
$\mathcal{J}_{43}$ & \vspace{0.2cm}
$e_{1}^{2}=e_{1}\quad n_{1}^{2}=n_{2}$

$e_{1}n_{1}=n_{1}\quad e_{1}n_{2}=n_{2}$ $e_{1}n_{3}=n_{3}\quad n_{1}n_{3}=n_{2}$\vspace{0.2cm}
 & $4$ & $0$ & $4$ & $N=\mathcal{T}_{4}$,

unitary,

associative\tabularnewline
\hline 
\end{tabular}

\caption{Four-dimensional Jordan algebras with three-dimensional radical of
type $(2,1)$ and $N=N_{1}$.}
\end{table}

\item $N=N_{\frac{1}{2}}$. Then $N^{2}\subseteq N_{0}\oplus N_{1}=0$,
but then $N$ would not have nilpotency type $(2,1)$. \vspace{0.2cm}

\item $N=N_{0}\oplus N_{\frac{1}{2}}$, $\dim N_{0}=2$. By Lemma \ref{lemma}
and \eqref{eq:pierce_one}, $N^{2}\subseteq N_{0}$. Consequently,
there exists $n_{2}\in N_{0}$ such that $N^{2}=\ka n_{2}$, choosing
some $n_{1}\in N_{\frac{1}{2}}$ and $n_{3}\in N_{0}$ we obtain $N_{0}=\ka n_{2}+\ka n_{3}$
with 
\begin{gather*}
n_{2}^{2},\ n_{2}n_{3}\in N^{3}=0,\ \ n_{1}n_{2}=n_{1}n_{3}=0\mbox{ by Lemma \ref{lemma}},\\
n_{3}^{2}=\alpha n_{2}\mbox{ and }n_{1}^{2}=\beta n_{2}\mbox{ since both belong to \ensuremath{N^{2}}}.
\end{gather*}
Depending on the values of $\alpha$ and $\beta$, we obtain the following
algebras:

\begin{table}[H]
\begin{tabular}{|c|>{\centering}m{5cm}|>{\centering}m{1.3cm}|>{\centering}m{1.3cm}|>{\centering}m{1cm}|>{\centering}m{3cm}|}
\hline 
$\J$  & Multiplication Table & $\dim$ $\operatorname*{Aut}(\J)$ & $\dim$ $\operatorname*{Ann}(\J)$ & $\dim$ $\J^{2}$ & Observation\tabularnewline
\hline 
$\mathcal{J}_{44}$ & \vspace{0.2cm}
$\mathcal{T}_{8}\oplus\ka n_{3}$\vspace{0.2cm}
 & $4$ & $2$ & $3$ & $\alpha=0,\beta\neq0$\tabularnewline
\hline 
$\mathcal{J}_{45}$ & \vspace{0.2cm}
$e_{1}^{2}=e_{1}\quad n_{1}^{2}=n_{2}$

$n_{3}^{2}=n_{2}\quad e_{1}n_{1}=\frac{1}{2}n_{1}$\vspace{0.2cm}
 & $3$ & $1$ & $3$ & $\alpha\neq0,\beta\neq0$\tabularnewline
\hline 
$\mathcal{J}_{46}$ & \vspace{0.2cm}
$\mathcal{B}_{2}\oplus\mathcal{B}_{3}$\vspace{0.2cm}
 & $4$ & $1$ & $3$ & $\alpha\neq0,\beta=0$\tabularnewline
\hline 
\end{tabular}

\caption{Four-dimensional Jordan algebras with three-dimensional radical of
type $(2,1)$ and $N=N_{0}\oplus N_{\frac{1}{2}}$ with $\dim N_{0}=2$.}
\end{table}

\noindent Note that if $\alpha=\beta=0$ then $N^{2}=0$. This contradicts
to the fact that $N$ has nilpotency type $(2,1)$.\vspace{0.2cm}

\item $N=N_{0}\oplus N_{1}$, $\dim N_{0}=2$. By Lemma \ref{lemma} and
\eqref{eq:pierce_one}, $N^{2}\subseteq N_{0}$. Consequently, there
exists $n_{2}\in N_{0}$ such that $N^{2}=\ka n_{2}$, choosing some
$n_{1}\in N_{0}$ and $n_{3}\in N_{1}$ we obtain $N_{0}=\ka n_{1}+\ka n_{2}$
with
\begin{gather*}
n_{2}^{2},\ n_{2}n_{1}\in N^{3}=0,\ \ n_{3}^{2}=n_{2}n_{3}=n_{1}n_{3}=0\mbox{,}\\
n_{1}^{2}=\alpha n_{2}\mbox{ since \ensuremath{n_{1}^{2}\in N^{2}}}.
\end{gather*}
If $\alpha=0$ then $N^{2}=0$, this contradicts to the fact that
$N$ has nilpotency type $(2,1)$. Thus $\alpha\neq0$ and we obtain
the following associative algebra:

\begin{table}[H]
\begin{tabular}{|c|>{\centering}m{5cm}|>{\centering}m{1.3cm}|>{\centering}m{1.3cm}|>{\centering}m{1cm}|>{\centering}m{3cm}|}
\hline 
$\J$ & Multiplication Table & $\dim$ $\operatorname*{Aut}(\J)$ & $\dim$ $\operatorname*{Ann}(\J)$ & $\dim$ $\J^{2}$  & Observation\tabularnewline
\hline 
$\mathcal{J}_{47}$ & \vspace{0.2cm}
$\mathcal{B}_{1}\oplus\mathcal{B}_{3}$\vspace{0.2cm}
 & $3$ & $1$ & $3$ & associative\tabularnewline
\hline 
\end{tabular}

\caption{Four-dimensional Jordan algebra with three-dimensional radical of
type $(2,1)$ and $N=N_{0}\oplus N_{1}$ with $\dim N_{0}=2$.}
\end{table}

\item $N=N_{0}\oplus N_{\frac{1}{2}}$, $\dim N_{\frac{1}{2}}=2$. Consider
$n_{1}\in N_{0}$, $n_{2},\mbox{ }n_{3}\in N_{\frac{1}{2}}$. By Lemma
\ref{lemma} it follows that $n_{1}^{2}=0$. On the other hand $n_{2}^{2}\mbox{, }n_{3}^{2},\mbox{ }n_{2}n_{3}\in N_{\frac{1}{2}}^{2}\subseteq N_{0}$
thus $n_{2}^{2}=\alpha n_{1}$, $n_{3}^{2}=\beta n_{1}$ and $n_{2}n_{3}=\gamma n_{1}$.
Since $n_{1}n_{2},\mbox{ }n_{1}n_{3}\in N_{0}N_{\frac{1}{2}}\subseteq N_{\frac{1}{2}},$
we have $n_{1}n_{2}=\theta n_{2}+\delta n_{3}$ and $n_{1}n_{3}=\theta_{1}n_{2}+\delta_{1}n_{3}$.

Next we check that $\left\{ e,\, n_{1},\, n_{2},\, n_{3}\right\} $
satisfy the Jordan identity \eqref{eq:linearjord} and then we obtain
conditions for the constants. In all the cases the algebras we obtain
are isomorphic to one of the following algebras:

\begin{table}[H]
\begin{tabular}{|c|>{\centering}m{5cm}|>{\centering}m{1.3cm}|>{\centering}m{1.3cm}|>{\centering}m{1cm}|>{\centering}m{3cm}|}
\hline 
$\J$ & Multiplication Table & $\dim$ $\operatorname*{Aut}(\J)$ & $\dim$ $\operatorname*{Ann}(\J)$ & $\dim$ $\J^{2}$ & Observation\tabularnewline
\hline 
$\mathcal{J}_{48}$ & \vspace{0.3cm}
$e_{1}^{2}=e_{1}\quad n_{3}^{2}=n_{1}$

$e_{1}n_{2}=\frac{1}{2}n_{2}\quad e_{1}n_{3}=\frac{1}{2}n_{3}$\vspace{0.3cm}
 & $5$ & $1$ & $4$ & \tabularnewline
\hline 
$\mathcal{J}_{49}$ & \vspace{0.3cm}
$e_{1}^{2}=e_{1}\quad n_{3}^{2}=n_{1}$

$e_{1}n_{2}=\frac{1}{2}n_{2}\quad e_{1}n_{3}=\frac{1}{2}n_{3}$

$n_{2}n_{3}=n_{1}$\vspace{0.3cm}
 & $4$ & $1$ & $4$ & \tabularnewline
\hline 
$\mathcal{J}_{50}$ & \vspace{0.3cm}
$e_{1}^{2}=e_{1}\quad e_{1}n_{2}=\frac{1}{2}n_{2}$ $e_{1}n_{3}=\frac{1}{2}n_{3}\quad n_{1}n_{2}=n_{3}$\vspace{0.3cm}
 & $5$ & $0$ & $3$ & \tabularnewline
\hline 
\end{tabular}

\caption{Four-dimensional Jordan algebras with three-dimensional radical of
type $(2,1)$ and $N=N_{0}\oplus N_{\frac{1}{2}}$ with $\dim N_{\frac{1}{2}}=2$.}
\end{table}

\item $N=N_{0}\oplus N_{\frac{1}{2}}\oplus N_{1}$, $\dim N_{i}=1$, for
$i=0,\,\frac{1}{2},\,1$. Consider $n_{1}\in N_{\frac{1}{2}}$, $n_{2}\in N_{1}$,
$n_{3}\in N_{0}$. We have the following products in $N$: 
\begin{gather*}
n_{3}^{2}=n_{2}^{2}=n_{1}n_{3}=n_{1}n_{2}=0\mbox{ by Lemma \ref{lemma}},\ \ n_{2}n_{3}\in N_{0}N_{1}=0,\\
n_{1}^{2}=\alpha n_{3}+\beta n_{2}\mbox{ since \ensuremath{n_{1}^{2}\in N_{\frac{1}{2}}^{2}\subseteq N_{0}\oplus N_{1}.}}
\end{gather*}
Depending on the values of $\alpha$ and $\beta$ we obtain the following
algebras:

\begin{table}[H]
\begin{tabular}{|c|>{\centering}m{5cm}|>{\centering}m{1.3cm}|>{\centering}m{1.3cm}|>{\centering}m{1cm}|>{\centering}m{3cm}|}
\hline 
$\J$ & Multiplication Table & $\dim$ $\operatorname*{Aut}(\J)$  & $\dim$ $\operatorname*{Ann}(\J)$  & $\dim$ $\J^{2}$  & Observation\tabularnewline
\hline 
$\mathcal{J}_{51}$ & \vspace{0.1cm}
$\mathcal{T}_{9}\oplus\ka n_{3}$\vspace{0.1cm}
 & $3$ & $1$ & $3$ & $\alpha=0$

$\beta\neq0$\tabularnewline
\hline 
$\mathcal{J}_{52}$ & \vspace{0.1cm}
$e_{1}^{2}=e_{1}\quad n_{1}^{2}=n_{3}$

$e_{1}n_{2}=n_{2}\quad e_{1}n_{1}=\frac{1}{2}n_{1}$\vspace{0.1cm}
 & $3$ & $1$ & $4$ & $\alpha\neq0$

$\beta=0$\tabularnewline
\hline 
$\mathcal{J}_{53}$ & \vspace{0.1cm}
$e_{1}^{2}=e_{1}\quad n_{1}^{2}=n_{3}+n_{2}$

$e_{1}n_{2}=n_{2}\quad e_{1}n_{1}=\frac{1}{2}n_{1}$\vspace{0.1cm}
 & $2$ & $1$ & $4$ & $\alpha\neq0$

$\beta\neq0$\tabularnewline
\hline 
\end{tabular}

\caption{Four-dimensional Jordan algebras with three-dimensional radical of
type $(2,1)$ and $N=N_{0}\oplus N_{\frac{1}{2}}\oplus N_{1}$.}
\end{table}

Note that if $\alpha=\beta=0$ then $N$$^{2}=0$ which is a contradiction
to the fact that $N$ has nilpotency type $(2,1)$.

\item $N=N_{0}\oplus N_{1}$, $\dim N_{1}=2$. By Lemma \ref{lemma} and
\eqref{eq:pierce_one} it follows that $N^{2}\subseteq N_{1}$. Consequently,
there exists $n_{2}\in N_{1}$ such that $N^{2}=\ka n_{2}$, choosing
some $n_{1}\in N_{1}$ and $n_{3}\in N_{0}$ we obtain $N_{1}=\ka n_{1}+\ka n_{2}$
with
\begin{gather*}
n_{2}^{2},\ n_{1}n_{2}\in N^{3}=0,\ \ n_{3}^{2}=n_{2}n_{3}=n_{1}n_{3}=0,\\
n_{1}^{2}=\alpha n_{2}\mbox{ since \ensuremath{n_{1}^{2}\in N^{2}}}.
\end{gather*}
If $\alpha=0$ then $N^{2}=0$ which is a contradiction to the fact
that $N$ has nilpotency type $(2,1)$. Consequently $\alpha\neq0$
and we obtain the following algebra:

\begin{table}[H]
\begin{tabular}{|c|>{\centering}m{5cm}|>{\centering}m{1.3cm}|>{\centering}m{1.3cm}|>{\centering}m{1cm}|>{\centering}m{3cm}|}
\hline 
$\J$ & Multiplication Table & $\dim$ $\operatorname*{Aut}(\J)$ & $\dim$ $\operatorname*{Ann}(\J)$ & $\dim$ $\J^{2}$ & Observation\tabularnewline
\hline 
$\mathcal{J}_{54}$  & \vspace{0.1cm}
$\mathcal{T}_{1}\oplus\ka n_{3}$\vspace{0.1cm}
 & $3$ & $1$ & $3$ & associative\tabularnewline
\hline 
\end{tabular}

\caption{Four-dimensional Jordan algebra with three-dimensional radical of
type $(2,1)$ and $N=N_{0}\oplus N_{1}$ with $\dim N_{1}=2$.}
\end{table}

\item $N=N_{\frac{1}{2}}\oplus N_{1}$, $\dim N_{1}=2$. By Lemma \ref{lemma}
and \eqref{eq:pierce_one} it follows that $N^{2}\subseteq N_{1}$.
Consequently there exists $n_{2}\in N_{1}$ such that $N^{2}=\ka n_{2}$,
choosing some $n_{1}\in N_{1}$ and $n_{3}\in N_{\frac{1}{2}}$ we
obtain $N_{1}=\ka n_{1}+\ka n_{2}$ with: $\dot{n_{2}^{2},\ n_{2}n_{1}\in N^{3}=0}$,
$n_{2}n_{3}=n_{1}n_{3}=0$ by Lemma \ref{lemma}, $n_{1}^{2}=\alpha n_{2}$
and $n_{3}^{2}=\beta n_{2}$ since both belong to $N^{2}$.

Depending on the values of $\alpha$ and $\beta$ we obtain the following
algebras:

\begin{table}[H]
\begin{tabular}{|c|>{\centering}m{5cm}|>{\centering}m{1.3cm}|>{\centering}m{1.3cm}|>{\centering}m{1cm}|>{\centering}m{3cm}|}
\hline 
$\J$ & Multiplication Table & dim $\operatorname*{Aut}(\J)$ & $\dim$ $\operatorname*{Ann}(\J)$ & $\dim$ $\J^{2}$ & Observation\tabularnewline
\hline 
$\mathcal{J}_{55}$ & \vspace{0.05cm}
$e_{1}^{2}=e_{1}\quad n_{3}^{2}=n_{2}$

$e_{1}n_{1}=n_{1}$

$e_{1}n_{2}=n_{2}\quad e_{1}n_{3}=\frac{1}{2}n_{3}$\vspace{0.05cm}
 & $4$ & $0$ & $4$ & $\alpha=0$

$\beta\neq0$\tabularnewline
\hline 
$\mathcal{J}_{56}$ & \vspace{0.05cm}
$e_{1}^{2}=e_{1}\quad n_{1}^{2}=n_{2}$

$e_{1}n_{1}=n_{1}$

$e_{1}n_{2}=n_{2}\quad e_{1}n_{3}=\frac{1}{2}n_{3}$\vspace{0.05cm}
 & $4$ & $0$ & $4$ & $\alpha\neq0$

$\beta=0$\tabularnewline
\hline 
$\mathcal{J}_{57}$ & \vspace{0.05cm}
$e_{1}^{2}=e_{1}\quad n_{1}^{2}=n_{2}$

$n_{3}^{2}=n_{2}\quad e_{1}n_{1}=n_{1}$ $e_{1}n_{2}=n_{2}\quad e_{1}n_{3}=\frac{1}{2}n_{3}$\vspace{0.05cm}
 & $3$ & $0$ & $4$ & $\alpha\neq0$

$\beta\neq0$\tabularnewline
\hline 
\end{tabular}

\caption{Four-dimensional Jordan algebras with three-dimensional radical of
type $(2,1)$ and $N=N_{\frac{1}{2}}\oplus N_{1}$ with $\dim N_{1}=2$.}
\end{table}

Note that if $\alpha=\beta=0$ then $N^{2}=0$. This contradicts to
the fact that $N$ has nilpotency type $(2,1)$.

\item $N=N_{\frac{1}{2}}\oplus N_{1}$, $\dim N_{\frac{1}{2}}=2$. Consider
$n_{1}\in N_{1}$, $n_{2},\mbox{ }n_{3}\in N_{\frac{1}{2}}$. By Lemma
\ref{lemma} it follows that $n_{1}^{2}=0$. On the other hand $n_{2}^{2}\mbox{, }n_{3}^{2},\mbox{ }n_{2}n_{3}\in N_{\frac{1}{2}}^{2}\subseteq N_{1}$
thus $n_{2}^{2}=\alpha n_{1}$, $n_{3}^{2}=\beta n_{1}$ and $n_{2}n_{3}=\gamma n_{1}$.
Since $n_{1}n_{2},\mbox{ }n_{1}n_{3}\in N_{1}N_{\frac{1}{2}}\subseteq N_{\frac{1}{2}}$,
we have $n_{1}n_{2}=\theta n_{2}+\delta n_{3}$ and $n_{1}n_{3}=\theta_{1}n_{2}+\delta_{1}n_{3}$.

Next we check that $\left\{ e_{1},\, n_{1},\, n_{2},\, n_{3}\right\} $
satisfy the Jordan identity \eqref{eq:linearjord} and then we obtain
conditions for the constants. In all the cases the algebras we obtain
are isomorphic to one of the following algebras:

\begin{table}[H]
\begin{tabular}{|c|>{\centering}m{5cm}|>{\centering}m{1.3cm}|>{\centering}m{1.3cm}|>{\centering}m{1cm}|>{\centering}m{3cm}|}
\hline 
$\J$ & Multiplication Table & $\dim$ $\operatorname*{Aut}(\J)$ & $\dim$ $\operatorname*{Ann}(\J)$ & $\dim$ $\J^{2}$ & Observation\tabularnewline
\hline 
$\mathcal{J}_{58}$ & \vspace{0.05cm}
$e_{1}^{2}=e_{1}\quad n_{3}^{2}=n_{1}$

$e_{1}n_{1}=n_{1}$

$e_{1}n_{2}=\frac{1}{2}n_{2}\quad e_{1}n_{3}=\frac{1}{2}n_{3}$\vspace{0.05cm}
 & $5$ & $0$ & $4$ & \tabularnewline
\hline 
$\mathcal{J}_{59}$ & \vspace{0.05cm}
$e_{1}^{2}=e_{1}\quad n_{3}^{2}=n_{1}$

$e_{1}n_{1}=n_{1}\quad e_{1}n_{2}=\frac{1}{2}n_{2}$

$e_{1}n_{3}=\frac{1}{2}n_{3}\quad n_{2}n_{3}=n_{1}$\vspace{0.05cm}
 & $4$ & $0$ & $4$ & \tabularnewline
\hline 
$\mathcal{J}_{60}$ & \vspace{0.05cm}
$e_{1}^{2}=e_{1}\quad e_{1}n_{1}=n_{1}$ $e_{1}n_{2}=\frac{1}{2}n_{2}\quad e_{1}n_{3}=\frac{1}{2}n_{3}$
$n_{1}n_{2}=n_{3}$\vspace{0.05cm}
 & $5$ & $0$ & $4$ & \tabularnewline
\hline 
\end{tabular}

\caption{Four-dimensional Jordan algebras with three-dimensional radical of
type $(2,1)$ and $N=N_{\frac{1}{2}}\oplus N_{1}$ with $\dim N_{\frac{1}{2}}=2$.}
\end{table}

\end{enumerate}

\subsection{Nilpotent Jordan algebras}

Four-dimensional nilpotent Jordan algebras over the field of complex
numbers were described in \cite{ancocheabermudes}. Observe that this
classification is still valid for any algebraically closed field of
$\car\neq2$. We have the following pairwise nonisomorphic algebras:

\begin{table}[H]
\begin{tabular}{|c|>{\centering}m{5cm}|>{\centering}m{1.3cm}|>{\centering}m{1.3cm}|>{\centering}m{1cm}|>{\centering}m{4.5cm}|}
\hline 
$\J$ & Multiplication Table & $\dim$ $\operatorname*{Aut}(\J)$ & $\dim$ $\operatorname*{Ann}(\J)$ & $\dim$ $\J^{2}$ & Observation\tabularnewline
\hline 
$\mathcal{J}_{61}$ & \vspace{0.2cm}
$n_{1}^{2}=n_{2}\quad n_{2}^{2}=n_{4}$

$n_{1}n_{2}=n_{3}\quad n_{1}n_{3}=n_{4}$\vspace{0.2cm}
 & $4$ & $1$ & $3$ & associative, 

nilpotency type $(1,1,1,1)$\tabularnewline
\hline 
$\mathcal{J}_{62}$ & \vspace{0.2cm}
$n_{1}^{2}=n_{2}\quad n_{4}^{2}=n_{2}$

$n_{1}n_{2}=n_{3}$\vspace{0.2cm}
 & $3$ & $1$ & $2$ & nilpotency type $(2,1,1)$\tabularnewline
\hline 
$\mathcal{J}_{63}$ & \vspace{0.2cm}
$n_{1}^{2}=n_{2}\quad n_{4}^{2}=-n_{2}-n_{3}$ $n_{1}n_{2}=n_{3}\quad n_{2}n_{4}=n_{3}$\vspace{0.2cm}
 & $4$ & $1$ & $2$ & nilpotency type $(2,1,1)$\tabularnewline
\hline 
$\mathcal{J}_{64}$ & \vspace{0.2cm}
$n_{1}^{2}=n_{2}\quad n_{4}^{2}=-n_{2}$ $n_{1}n_{2}=n_{3}\quad n_{2}n_{4}=n_{3}$\vspace{0.2cm}
 & $5$ & $1$ & $2$ & nilpotency type $(2,1,1)$\tabularnewline
\hline 
$\mathcal{J}_{65}$ & \vspace{0.2cm}
$n_{1}^{2}=n_{2}\quad n_{1}n_{2}=n_{3}$

$n_{2}n_{4}=n_{3}$\vspace{0.2cm}
 & $4$ & $1$ & $2$ & nilpotency type $(2,1,1)$\tabularnewline
\hline 
$\mathcal{J}_{66}$ & \vspace{0.2cm}
$n_{1}^{2}=n_{2}\quad n_{4}^{2}=n_{3}$

$n_{1}n_{2}=n_{3}$\vspace{0.2cm}
 & $5$ & $1$ & $2$ & associative, 

nilpotency type $(2,1,1)$\tabularnewline
\hline 
$\mathcal{J}_{67}$ & $\mathcal{T}_{3}\oplus\ka n_{4}$ & $6$ & $2$ & $2$ & associative, 

nilpotency type $(2,1,1)$\tabularnewline
\hline 
$\mathcal{J}_{68}$ & $\mathcal{B}_{3}\oplus\mathcal{B}_{3}$ & $6$ & $2$ & $2$ & associative, 

nilpotency type $(2,2)$\tabularnewline
\hline 
$\mathcal{J}_{69}$ & \vspace{0.2cm}
$n_{1}^{2}=n_{2}$

$n_{1}n_{3}=n_{4}$\vspace{0.2cm}
 & $7$ & $2$ & $2$ & associative, 

nilpotency type $(2,2)$\tabularnewline
\hline 
$\mathcal{J}_{70}$ & \vspace{0.2cm}
$n_{1}^{2}=n_{2}$

$n_{3}n_{4}=n_{2}$\vspace{0.2cm}
 & $7$ & $1$ & $1$ & associative, 

nilpotency type $(3,1)$\tabularnewline
\hline 
$\mathcal{J}_{71}$ & $\mathcal{T}_{4}\oplus\ka n_{4}$ & $8$ & $2$ & $1$ & associative, 

nilpotency type $(3,1)$\tabularnewline
\hline 
$\mathcal{J}_{72}$ & $\mathcal{B}_{3}\oplus\ka n_{3}\oplus\ka n_{4}$ & $10$ & $3$ & $1$ & associative, 

nilpotency type $(3,1)$\tabularnewline
\hline 
$\mathcal{J}_{73}$ & $\ka n_{1}\oplus\ka n_{2}\oplus\ka n_{3}\oplus\ka n_{4}$ & $16$ & $4$ & $0$ & associative, 

nilpotency type $(4)$\tabularnewline
\hline 
\end{tabular}

\caption{Four-dimensional nilpotent Jordan algebras.}
\end{table}

\section{Remarks\label{sec:Remarks}}

\subsection{All algebras in the previous section are pairwise nonisomorphic.}

Comparing the algebra invariants in Section \ref{sec:4-Jordan-algebras},
namely $\dim\operatorname*{Ann}(\J)$, $\dim\operatorname*{Aut}(\J)$
and $\dim\J^{2}$, together with properties whether $\J$ is indecomposable,
associative, nonassociative, unitary, one needs only to verify whether
there exist isomorphisms between $\J_{55}$,$\mbox{ }\J_{56}$,$\mbox{ }\J_{59}$
and whether there exists isomorphism between $\J_{58}$ and $\J_{60}$.

First observe that the two-dimensional nonassociative Jordan algebra
$\mathcal{B}_{2}$ is a subalgebra of $\J_{56}$ but it is not a subalgebra
of $\J_{55}$, then $\J_{55}\not\simeq\J_{56}$.

Further, the second cohomology group $H^{2}(\J_{59},\J_{59})=0$ while
both $H^{2}(\J_{55},\J_{55})$ and $H^{2}(\J_{56},\J_{56})$ are not
zero, thus we have $\J_{59}\not\simeq\J_{55}$ and $\J_{59}\not\simeq\J_{56}$.
For the definition of this cohomology group for Jordan algebras we
refer to \cite[Chap.II, sect. 8]{jacobson}.

Finally, $\operatorname*{Rad}(\J_{58})=\mathcal{B}_{3}\oplus\ka n_{2}$
while $\operatorname*{Rad}(\J_{60})=\mathcal{T}_{4}$, therefore $\J_{58}\not\simeq\J_{60}$.

\subsection{All algebras in the previous sections are special.}

In 1979, A. M. Slin'ko in \cite{slinko} proved that nilpotent Jordan
algebras up to dimension $5$ are special and in 1989, H. Sherkulov
in \cite{TeseRuso} showed that nonassociative Jordan algebras up
to dimension $4$ are special.

\section{Acknowledgement}

The author is grateful to Iryna Kashuba and Daniel Miranda for their
careful reading of the paper and for many useful comments.

The author was supported by the CAPES scholarship for PhD program
of Mathematics, IME-USP.

\end{document}